\documentclass[11pt,a4paper]{article}
\usepackage{pgf}
\usepackage{tikz}
\usepackage[symbol]{footmisc}
\usetikzlibrary{arrows,automata}
\usepackage{verbatim}
\usepackage{caption}
\usepackage{epsf,epsfig,amsfonts,amsgen,amsmath,amstext,amsbsy,amsopn,amsthm,amssymb}
\usepackage{color}
\usepackage{mathrsfs,hyperref}
\usepackage{enumerate}
\usepackage{url}
\usepackage{enumitem}
\usetikzlibrary{patterns}

\usepackage{setspace}
\newtheorem{dfn}{Definition}[section]
\newtheorem{thm}[dfn]{Theorem}
\newtheorem{thmo}{Theorem}[section]

\newtheorem{clm}[dfn]{Claim}

\newtheorem{lem}[dfn]{Lemma}

\newtheorem{prop}[thmo]{Proposition}

\newtheorem{conjecture}[dfn]{Conjecture}

\def\QED{\hfill \rule{7pt}{7pt}}

\newcommand{\dH}{{\mathcal{H}}}

\newcommand{\dE}{{\mathcal{E}}}

\newcommand{\dC}{{\mathcal{C}}}

\newcommand{\dF}{{\mathcal{F}}}

\newcommand{\dN}{{\mathcal{N}}}

\newcommand{\dS}{{\mathcal{S}}}

\newcommand{\tcoveropt}[1]{\mathcal{C}\mathcal{C}_t(#1)}

\newcommand{\turangraph}[2]{T_{#1,#2}}
\newcommand{\coveropt}[2]{{\mathcal{C}\mathcal{C}_{#1}(#2)}}

\makeatletter

\newcommand{\Rmnum}[1]{\expandafter\@slowromancap\romannumeral #1@}
\makeatother
\begin{document}
\title{On the $4$-clique cover number of graphs}

\author{Yihan Chen\footnotemark[1]
\and Jialin He\footnotemark[2]
\and
	Tianying Xie\footnotemark[3]}
\date{}

\maketitle

\footnotetext[1]{School of Mathematical Sciences, University of Science and Technology of China, Hefei, Anhui, 230026, China. Email: cyh2020@mail.ustc.edu.cn}

\footnotetext[2]{Department of Mathematics, Hong Kong University of Science and Technology, Clear Water Bay, Kowloon 999077, Hong Kong. Email: majlhe@ust.hk}

\footnotetext[3]{School of Mathematical Sciences, University of Science and Technology of China, Hefei, Anhui, 230026, China. Email: xiety@ustc.edu.cn}

\begin{abstract}
    In 1966, Erd\H{o}s, Goodman, and P{\'o}sa proved that $\lfloor n^2/4 \rfloor$ cliques are sufficient to cover all edges in any $n$-vertex graph, with tightness achieved by the balanced complete bipartite graph. 
    This result was generalized by Dau, Milenkovic, and Puleo, who showed that at most 
    $\lfloor \frac n 3 \rfloor \lfloor \frac {n+1} 3 \rfloor \lfloor \frac {n+2} 3 \rfloor$ 
    cliques are needed to cover all triangles in any $n$-vertex graph $G$, and the bound is best possible as witnessed by the balanced complete tripartite graph. 
    They further conjectured that for $t \geq 4$, the $t$-clique cover number is maximized by the Tur\'an graph $\turangraph{n}{t}$. 
    We confirm their conjecture for $t=4$ using novel techniques, including inductive frameworks, greedy partition method, local adjustments, and clique-counting lemmas by Erd\H{o}s and by Moon and Moser.\\

\end{abstract}

\section{Introduction}
For $n \ge t$, the \emph{Tur\'an graph} $\turangraph{n}{t}$ is the $n$-vertex complete $t$-partite graph such that the part sizes differ by at most one.
A graph is \textit{$K_t$-free} if it contains no copy of the $t$-clique $K_t$ as a subgraph.
Tur\'an's celebrated theorem~\cite{Turan} states that $\turangraph{n}{t-1}$ is the unique $n$-vertex $K_t$-free graph maximizing the number of edges. 
Generalizations of Tur\'an's theorem in various directions form a central theme in extremal graph theory.
Our starting point is one of the early result in this area due to Erd\H{o}s, Goodman, and P\'osa~\cite{EGP}, which says that the edge set of every $n$-vertex graph $G$ can be covered by at most $\lfloor n^2/4 \rfloor$ cliques, and equality holds if and only if $G$ is $\turangraph{n}{2}.$
Equivalently, the edges of a triangle-free graph with maximized number of edges are the hardest to cover with cliques.
 
Inspired by community detection, Dau, Milenkovic, and Puleo~\cite{DMP} generalized the edge cover problem to the $t$-clique cover problem for $t > 2$.
A \emph{$t$-clique cover} of a graph $G$ is a collection of cliques such that every $t$-clique in $G$ is a subgraph of some clique of the collection. 
The \emph{$t$-clique cover number} of $G$, denoted by $\tcoveropt{G}$, is the minimum number of cliques in a $t$-clique cover of $G$.
Note that $\coveropt{1}{G}$ is the chromatic number of the complement of $G$. 
The graph parameter $\coveropt{2}{G}$ is well-studied and also known as the intersection number of $G$~\cite{EGP,L68,R85,ST99}. 
Dau et al.~\cite{DMP} extended the Erd\H{o}s-Goodman-P\'osa result~\cite{EGP} to $t=3$, 
proving that $\mathcal{C}\mathcal{C}_3(G)\le \mathcal{C}\mathcal{C}_3(\turangraph{n}{3})$ for all $n$-vertex graph $G$ and (when $n \ge 3$) equality holds if and only if $G$ is $\turangraph{n}{3}$. 
They conjectured the following generalization for $t\ge4$.

%We are interested in the extremal problem of maximizing $\tcoveropt{G}$ over all $n$-vertex graphs $G$.
%Erd{\H o}s, Goodman and P{\'o}sa~\cite{EGP} proved that for every $n$-vertex graph $G$, $\mathcal{C}\mathcal{C}_2(G) \le \lfloor\frac{n^2}{4}\rfloor$ 
%with equality given by the balanced complete bipartite graph. 
%Lov\'asz~\cite{L68} simultaneously got the same upper bound of $\coveropt{2}{G}$ by partitioning the vertex set of $G$ into disjoint union of cliques and greedily constructing a $2$-clique cover according to this partition (see Section~\ref{Subsec:Greedy partition and key lemma} for more details of this structure).

\begin{conjecture}[Dau, Milenkovic and Puleo~\cite{DMP}, Conjecture 1]\label{conj:clique_cover}
If $n$ and $t$ are positive integers such that $4\le t\le n$, 
then for every $n$-vertex graph $G$, we have
\[ \tcoveropt{G} \le \tcoveropt{\turangraph{n}{t}}, \]
with equality if and only if $G = \turangraph{n}{t}$. 
\end{conjecture}

The value $\tcoveropt{\turangraph{n}{t}}$ is easily calculated which is equal to the number of $K_t$'s in $\turangraph{n}{t}$, since the largest cliques in $\turangraph{n}{t}$ are $K_t$'s, each must be taken in any $t$-clique cover of $\turangraph{n}{t}$. 
For $t=3$, Dau, Milenkovic, and Puleo~\cite{DMP} confirmed Conjecture~\ref{conj:clique_cover} by induction on $n$.
Recently, the second author, together with Balogh, Krueger, Nguyen and Wigal~\cite{BHKNW}, proved an asymptotic version of Conjecture~\ref{conj:clique_cover} with a different approach, by showing that for every fixed $t \geq 2$, 
$\tcoveropt{G} \leq (1+o(1))\tcoveropt{\turangraph{n}{t}}$ 
for every $n$-vertex graph $G$, where the $o(1)$ goes to $0$ as $n$ goes to infinity.
For larger $t$, the case analysis in Dau et al.'s proof~\cite{DMP} becomes intractable. 

Our main contribution is the confirmation of Conjecture~\ref{conj:clique_cover} for $t=4$ using novel methods.

\begin{thm}	\label{Thm:Main}
    For every $n$-vertex graph $G$ with $n\ge 4,$ we have
	$\coveropt{4}{G} \leq \coveropt{4}{\turangraph{n}{4}},$
	with equality if and only if $G =\turangraph{n}{4}.$
\end{thm}

The proof of Theorem~\ref{Thm:Main} is based on an inductive framework on the number of vertices, 
using a generalization of Lov\'{a}sz's greedy partition approach~\cite{L68},
combining local adjustment operations and several clique-counting lemmas.
For a brief overview of the proof, we direct readers to the beginning of Section~3 and Section~4.

The paper is organized as follows.
Section~2 presents preliminaries, establishes clique number bounds via induction. 
%and verifies base cases for Theorem~\ref{Thm:Main} to streamline subsequent arguments.
In Section~3, we prove Theorem~\ref{Thm:Main} by reducing it to the key Lemma~\ref{Lem:key lemma}.
The complete proof of Lemma~\ref{Lem:key lemma} will be presented in Section~4.

\section{Preliminaries}\label{Sec:Pre}
This section is divided into two parts. 
First, we give notation and collect essential lemmas. 
Second, we use induction to bound the clique number.
%and verify Theorem~\ref{Thm:Main} for small $n$.

\subsection{Notations and lemmas}
Let $G$ be a graph with vertex set $V(G)$, edge set $E(G)$, and $e(G) = |E(G)|$. 
For $v \in V(G)$, let $N_G(v)$ denote its neighbor and $d_G(v) = |N_G(v)|$ denote its degree. 
The minimum degree of $G$ is $\delta(G)$.
For $U\subset V(G),$ 
$G[U]$ is the subgraph of $G$ induced by $U$,
$G\backslash U$ is the subgraph obtained from $G$ by deleting vertices in $U$, and
$N_{G}(U):=\bigcap_{v\in U} N_{G}(v)$ is the set of common neighbors of vertices in $U$ in $G$.
For disjoint $U,W\subset V(G)$,
$E_{G}(U,W)$ is the set of edges in $E(G)$ between $U$ and $W$, and $e_{G}(U,W):=|E_{G}(U,W)|$.
Similarly, we let
$N_U(W):=\{u\in U: uw\in E(G)\ \text{for any}\ w\in W\}$
be the set of common neighbors of vertices in $W$ within $U$.
Let $d_{U}(W):=|N_{U}(W)|,$
$N_U[W]:=N_U(W)\cup W$,
and $T_U[W]:=G[N_U[W]].$
The clique number $\omega(G)$ is the size of a maximum clique of $G$. 
We use $k_{t}(G)$ to denote the number of $t$-clique in $G.$
For all above notations, we often drop the subscripts when they are clear from context.
Throughout this paper, the notation $\binom{x}{2}$ means the function $x(x-1)/2$ for all reals $x$.
For any positive integer $k$, we write $[k]$ as the set $\{1,2,...,k\}$.
%We omit floors and ceilings when they are not critical and make no effort to optimize some of the constants.

Note that $\mathcal{C}\mathcal{C}_4(\turangraph{n}{4})=k_4(\turangraph{n}{4})$ for all positive integers $n$. 
It is easy to check that, for $n\ge 6,$ 
\begin{align}\label{Equ: bound of k4(T(n,4))}
    \frac {(n-2)^2(n+2)^2}{256}\le k_4(\turangraph{n}{4})\le \frac{n^4}{256}\ \  \text{and}\ \ \frac{(n-1)^3}{64}\le k_4(\turangraph{n}{4})-k_4(\turangraph{n-1}{4})\le \frac{n^3}{64}.
\end{align}

Our proofs use the following two fundamental counting lemmas considering the number of cliques proved by Erd\H{o}s~\cite{E62}, and Moon and Moser~\cite{MM}, respectively. 
%proved by Erd\H{o}s~\cite{E62}, and Moon and Moser~\cite{MM} are counting the maximum number of $k$-clique in $K_{\ell}$-free graphs and considering the connection among $k_{t+1}(G),$ $k_t(G)$ and $k_{t-1}(G)$, respectively. 

\begin{lem}[Erd\H{o}s~\cite{E62}]\label{Lem:Erdos}
	Let $k, t$ be two positive integers with $k>t.$
	We define $h(n,k,t)$ to be the maximum number of $t$-clique among all $n$-vertex $K_{k}$-free graphs.
	Then,
	$$
	h(n,k,t)
	= k_t(\turangraph{n}{k-1}) 
	=\sum_{0\leq i_1 < \cdots <i_t \leq k-2}
	\prod_{r=1}^{t} \left\lfloor \frac{n+i_r}{k-1} \right\rfloor
	\leq \binom{k-1}{t} \left(\frac{n}{k-1}\right)^t.
	$$
\end{lem}

\begin{lem}[Moon and Moser~\cite{MM}]\label{Lem:MM}
	For any integer $t\geq 2$
	and any $n$-vertex graph $G$,
	$$ \frac{k_{t+1}(G)}{k_t(G)} \geq \frac{1}{t^2-1} \left( t^2 \frac{k_t(G)}{k_{t-1}(G)}-n \right).
	$$
\end{lem}

\subsection{Reducing to $\omega(G)=5$}\label{Sec: Reducing to omega(G)=5}
We prove Theorem~\ref{Thm:Main} by induction on $n = |V(G)|$.
In particular, we will show that $\omega(G)\le 5$ under the inductive hypothesis.
The base case $n \in \{4,5\}$ can be easily checked and 
$\coveropt{4}{G} \leq \coveropt{4}{\turangraph{n}{4}}$ with equality if and only if $G = \turangraph{n}{4}$. 
Assume that Theorem~\ref{Thm:Main} holds for all graphs on at most $n-1$ vertices ($n \geq 6$), and let $G$ be an $n$-vertex graph.

Suppose that $\omega(G)=c$ for some integer $1\le c\le n$, 
and let $C$ be a $c$-clique in $G$.
We first claim that $c\ge 4.$
Indeed, if $c\le 3,$ 
then $\coveropt{4}{G}=0<\mathcal{C}\mathcal{C}_4(\turangraph{n}{4})$ for $n\ge 6.$

We give an upper bound of $c$ by constructing a $4$-clique cover of $G$ using $C$.  
Let $G'= G\setminus C,$
and let $\mathcal{C}_1'$ be a minimum $4$-clique cover of $G'.$
By induction hypothesis and inequality~\eqref{Equ: bound of k4(T(n,4))}, 
$$|\mathcal{C}_1'|=\mathcal{C}\mathcal{C}_4(G')\le k_4(\turangraph{n-c}{4})\le ((n-c)/4)^4.$$
We now extend $\mathcal{C}_1'$ to a $4$-clique cover of $G.$
Note that the $4$-cliques of $G$ not yet covered by $\mathcal{C}_1'$ are the $4$-cliques that contain at least one vertex of $C.$
Thus, we define
\begin{equation*}
\mathcal{C}_1:=
\mathcal{C}_1'\cup
\{C\} \cup
\{\bigcup_{v\in G'}T_C[v]\}
\cup \{\bigcup_{uv \in E(G')}T_C[u,v]\}
\cup \{\bigcup_{uvw \  is \ a \  K_3 \ in \  G'}T_C[u,v,w]\}. 
\end{equation*}
It is easy to see that $\dC_1$ is a $4$-clique cover of $G.$ 
Therefore, by using the obvious upper bounds of the number of edges and triangles in $G'$, we have
\begin{equation}\label{eq3.1}
\mathcal{C}\mathcal{C}_4(G)
\leq 
|\dC_1|\leq
\left(\frac{n-c}{4}\right)^4+1+(n-c)+ \binom{n-c}{2} + \binom{n-c}{3}.
\end{equation} 
Now we claim that Theorem~\ref{Thm:Main} holds when $c \ge 11.$
Indeed,  if $n\ge c\ge 11,$ since the right side of~\eqref{eq3.1} is decreasing with respect to $c,$
to complete the proof of Theorem~\ref{Thm:Main}, it suffices to show that
$\left(\frac{n-11}{4}\right)^4+1+(n-11)+ \binom{n-11}{2} + \binom{n-11}{3} < k_4(\turangraph{n}{4}).$
By~\eqref{Equ: bound of k4(T(n,4))}, we have
$$\left(\frac{n-11}{4}\right)^4+1+(n-11)+ \binom{n-11}{2} + \binom{n-11}{3}<
\frac{(n-2)^2(n+2)^2}{256}\le k_4(\turangraph{n}{4}),$$
where the first inequality is equivalent to 
$\frac{1}{3}n^3+\frac{337}{2}n^2-\frac{7783}{3}n+\frac{44255}{4}>0,$
which holds for all $n\ge 4.$
This implies that 
$\mathcal{C}\mathcal{C}_4(G)< \mathcal{C}\mathcal{C}_4(\turangraph{n}{4}),$
and we are done.
Therefore, in the following of the proof, we may assume that $4\le c\le 10.$

For $6\leq c \leq 10,$ we improve the bound from~\eqref{eq3.1}.
Since $C$ is a maximum clique, $G'$ is also $K_{c+1}$-free.
By Lemma~\ref{Lem:Erdos} (with $k=c+1$), the number of edges in $G'$ is at most 
$\binom{c}{2} \left(\frac{n-c}{c}\right)^2$ 
and the number of triangles in $G'$ is at most 
$\binom{c}{3} \left(\frac{n-c}{c}\right)^3$, which implies that
\begin{equation*}
\mathcal{C}\mathcal{C}_4(G)
\leq 
|\dC_1|\leq
\left(\frac{n-c}{4}\right)^4 +1 +(n-c)+\binom c2\left(\frac{n-c}{c}\right)^2+\binom{c}{3}\left(\frac{n-c}{c}\right)^3.
\end{equation*}
Therefore, we only need to show that
\begin{align}\label{Equ:bound of clique number}
    \left(\frac{n-c}{4}\right)^4 +1 +(n-c)+\binom c2\left(\frac{n-c}{c}\right)^2 +\binom{c}{3}\left(\frac{n-c}{c}\right)^3
< \frac{(n-2)^2(n+2)^2}{256}.
\end{align}
For $6\le c \le 10$ and $n \geq c$, \eqref{Equ:bound of clique number} holds (see Appendix A). 
Thus, we may assume that $4\le c\le 5.$

If $c = 4$, then $G$ is $K_5$-free. 
By Lemma~\ref{Lem:Erdos} (with $k=5,\ t=4$), we have
$\coveropt{4}{G} = k_4(G) \leq h(n,5,4) = k_4(\turangraph{n}{4})$, 
with equality if and only if $G =\turangraph{n}{4},$ which completes the proof of Theorem~\ref{Thm:Main}.

By the preceding analysis, 
to prove Theorem~\ref{Thm:Main} for any $n$-vertex graph $G$, 
we may assume that $\omega(G) = 5$ and the theorem holds for all graphs on fewer than $n$ vertices.

\section{Proof of Theorem~\ref{Thm:Main}}\label{Sec:Proof of main Theorem}
We first establish Theorem~\ref{Thm:Main} for $6 \leq n \leq 104$ with $n \not\in \{97,101\}$ (these base cases will simplify subsequent root computations and monotonicity checks; the proof is given in Appendix B).

Let $\mathcal{N}:= \{n\in \mathbb{N}|\ n\ge 105\ \text{or}\  n\in \{97,101\} \}.$
For the remainder of the proof, 
we suppose that $G$ is an $n$-vertex graph for some $n\in \dN$ such that $\omega(G)=5,$ and
Theorem~\ref{Thm:Main} holds for all graphs on fewer than $n$ vertices.

To prove Theorem~\ref{Thm:Main}, 
we use a similar strategy as before, which was also used in~\cite{DMP}: 
we select a vertex $v$ of minimum degree in $G,$ 
and consider the subgraphs $G\setminus\{v\}$ and $G[N(v)];$ 
we first use inductive hypothesis to bound the $4$-clique cover number of $G\setminus\{v\}$ and then extend it to a $4$-clique cover of $G$ by considering the uncovered $4$-cliques in $G[N(v)\cup\{v\}],$ i.e., the $3$-cliques in $G[N(v)].$
For $t=3$, Dau et al.~\cite{DMP} used Lov\'{a}sz's method~\cite{L68} to bound the edge cover number of $G[N(v)]$ in terms of $\delta(G)$. 
However, for $t=4$, bounding the 3-clique cover number of $G[N(v)]$ requires more than $\delta(G)$. 
To address this, we analyze $G[N(v)]$ using a greedy partition technique and local adjustments, obtaining our key lemma, Lemma~\ref{Lem:key lemma}, and then use it to complete the proof of Theorem~\ref{Thm:Main}.
\medskip

A crucial concept from Dau et al.~\cite{DMP} is the following special partition of vertex-disjoint cliques.  
We call it {\it greedy} because it can be obtained by iteratively applying the following greedy procedure: at each step, select a largest clique in the remaining graph and then delete the vertices of this clique.

\begin{dfn}
A \textit{greedy partition} of a graph $H$ of size $p$, denoted by $\dF_H:= \{A_1, \cdots, A_p\}$,
is a partition of $V(H)$ into $p$ disjoint cliques $A_i$ with maximum size in $H[V(H)\setminus (\cup_{j=1}^{i-1}A_j)]$ for $i\in [p].$ 
\end{dfn}

This greedy partition yields the following structural property (see Lemma 1 in~\cite{DMP}).

\begin{prop}\label{Claim:p<n-delta}
	Let $\dF_H= \{A_1, \cdots, A_p\}$ be a greedy partition of an $n$-vertex graph $H$ of size $p$.
	Then, for any $1\leq i<j \leq p,$
	every vertex in $A_j$ has a non-neighbor in $A_i.$ 
    Moreover, we have 
	$p \leq n-\delta(H).$
\end{prop}

Let $v\in V(G)$ be a vertex of minimum degree $d_v=\delta(G).$
Let $G'=G\setminus \{v\}$ and $H=G[N(v)].$
Since $\omega(G)=5,$ 
$G$ is $K_6$-free and thus $H$ is $K_5$-free.
By induction, there exists a $4$-clique cover $\mathcal{C}_1'$ of $G'$ with size 
$|\mathcal{C}_1'|=\mathcal{C}\mathcal{C}_4(G')\le k_4(\turangraph{n-1}{4}).$
Now, we extend $\mathcal{C}_1'$ to a $4$-clique cover of $G$ by adding cliques that cover all the $4$-cliques containing $v,$
which correspond exactly to covering all the triangles in $H$. 
Thus, the number of additional cliques needed is $\coveropt{3}{H}$,
which implies that

\begin{equation}\label{Equ:CC_4(G)<=CC_4(G')+CC_3(H)}
    \coveropt{4}{G}\le \coveropt{4}{G'}+\coveropt{3}{H}\le  k_4(\turangraph{n-1}{4})+ \coveropt{3}{H}.
\end{equation}

To bound $\coveropt{3}{H}$, we use the greedy partition $\mathcal{F}_H$ of $H$ to derive the following key lemma.

\begin{lem}\label{Lem:key lemma}
    Let $H$ be any $K_5$-free graph with $|V(H)| = m$ and $\mathcal{F}_H = \{A_1,\dots,A_p\}$ be any greedy partition of $H$ of size $p$. 
    Then there exists an integer $q$ with $m/4 \leq q \leq p$ such that       
    $$
    \coveropt{3}{H}\le \frac{59}{2}q^3 -24q^2m +\frac{13}{2}qm^2 -\frac{5}{9}m^3+Q_0,
    $$
where 
%where $Q_0 = \max\{Q_1, Q_2, Q_3\}$ and
%    \[
%    \begin{aligned}
%        Q_1 &= \frac{13}{2}q^2 - \frac{17}{6}qm + \frac{2}{9}m^2 + \frac{1}{3}q - \frac{1}{9}m + \frac{4}{9}, \\
%        Q_2 &= -23q^2 + \frac{79}{6}qm - \frac{35}{18}m^2 + \frac{11}{2}q - \frac{11}{6}m, \\
%        Q_3 &= -\frac{105}{2}q^2 + \frac{175}{6}qm - \frac{37}{9}m^2 + \frac{92}{3}q - \frac{80}%{9}m - \frac{16}{3}.
%    \end{aligned}
%    \]
$Q_0=\max\{Q_1,Q_2,Q_3\}$
    and 
         $Q_1=\frac{13}{2}q^2-\frac{17}{6}qm+\frac{2}{9}m^2+\frac{1}{3}q-\frac{1}{9}m+\frac{4}{9},$
         $Q_2=-23q^2+\frac{79}{6}qm-\frac{35}{18}m^2+\frac{11}{2}q-\frac{11}{6}m,$
         $Q_3=-\frac{105}{2}q^2+\frac{175}{6}qm-\frac{37}{9}m^2+\frac{92}{3}q-\frac{80}{9}m-\frac{16}{3}.$
\end{lem}

\subsection{Complete the proof of Theorem~\ref{Thm:Main}}\label{Subsec:proof of main thm}
We prove Theorem~\ref{Thm:Main} assuming Lemma~\ref{Lem:key lemma} holds, and postpone the proof of the lemma to Section~\ref{Sec:Proof of key lemma}.

\noindent
\textbf{Proof of Theorem \ref{Thm:Main}.}
By the analysis in Subsection~\ref{Sec: Reducing to omega(G)=5} and the beginning of Section~\ref{Sec:Proof of main Theorem}, 
$G$ is an $n$-vertex graph for some $n\in \dN$,
$\omega(G)=5,$ and
Theorem~\ref{Thm:Main} holds for all graphs on fewer than $n$ vertices.
To prove that Theorem~\ref{Thm:Main} holds for $G$, 
we need to show that 
$\coveropt{4}{G} < \coveropt{4}{\turangraph{n}{4}}=k_4(\turangraph{n}{4}).$ 
From the definitions of $v$ and $H$, by \eqref{Equ:CC_4(G)<=CC_4(G')+CC_3(H)}, it is sufficient to prove that

\begin{equation}\label{Equ:cc_3(H)<k_4(\turangraph{n}{4})-k_4(\turangraph{n-1}{4})}
    \coveropt{3}{H}< k_4(\turangraph{n}{4})-k_4(\turangraph{n-1}{4}).
\end{equation}

Since $v$ is a vertex of minimum degree, for every $w \in N(v),$ we have
$d_H(w) \geq d(w)-(n-d(v))\geq d(v)-(n-d(v)) =2d(v)-n.$
Thus $\delta(H) \geq 2d(v) -n.$
Let $\dF_H= \{A_1, \cdots, A_p\}$ be a greedy partition of $H$ of size $p$.
Note that $|V(H)|=d(v),$ by Proposition~\ref{Claim:p<n-delta}, we have

\begin{equation}\label{Equ:p<n-d(v)}
p\le |V(H)|-\delta(H)\le d(v)-(2d(v)-n)=n-d(v).
\end{equation}

For convenience, we use $d$ to denote $d(v)$ in the following of the proof.
We will complete the proof according to the range of values of $d.$
Since  $G$ is $K_6$-free, 
from the classical Tur\'{a}n Theorem,
$\frac 1 2 n d \le e(G)\leq \frac{2}{5}n^2,$
which implies that $d\leq \frac{4}{5}n.$
\medskip

If $d\leq \lfloor \frac{3}{4}n \rfloor-1,$ 
by the result of Dau et al. (see Theorem 3 in~\cite{DMP}) and Lemma~\ref{Lem:Erdos},
$$
\coveropt{3}{H}\leq \coveropt{3}{\turangraph{|H|}{3}}\le 
\frac{|V(H)|^3}{27}  
\leq \frac{\left(\left\lfloor \frac{3n}{4} \right\rfloor-1\right)^3}{27}  
< \frac{(n-1)^3}{64}
\le k_4(\turangraph{n}{4})-k_4(\turangraph{n-1}{4}),
$$
where the fourth inequality holds for $n\ge 1,$
the last inequality holds for $n\ge 4,$ and we are done.

If $d = \lfloor \frac{3}{4}n \rfloor,$ 
similarly by the above process again, we have
$$
\coveropt{3}{H}\leq \coveropt{3}{\turangraph{|H|}{3}}=k_3(\turangraph{|H|}{3}) 
= k_4(\turangraph{n}{4})-k_4(\turangraph{n-1}{4}),
$$
the last equality holds because for any $n$
we have $k_4(\turangraph{n}{4})-k_4(\turangraph{n-1}{4})=k_3\left(\turangraph{\lfloor \frac{3}{4}n \rfloor}{3}\right).$ 
By \eqref{Equ:CC_4(G)<=CC_4(G')+CC_3(H)}, we have 
$\coveropt{4}{G}\le \coveropt{4}{G'}+\coveropt{3}{H}\le \coveropt{4}{\turangraph{n}{4}}.$
However, if all the equalities hold, then $G$ would be $\turangraph{n}{4},$ contradicting $\omega(G)=5.$
Therefore, we obtain $\coveropt{4}{G}< \coveropt{4}{\turangraph{n}{4}},$ and we are done.

If $\lfloor \frac{3}{4}n \rfloor +1\leq d \leq \frac{4}{5}n,$
then 
$\frac{5}{4}d\le n< \frac 43d.$
By Lemma~\ref{Lem:key lemma} and \eqref{Equ:p<n-d(v)}, there is an integer $q$ satisfying
\begin{align}\label{Equ:range of q}
    \frac 1 4 d\le q\le p\le n-d<\frac 1 3 d,
\end{align}
such that
$$\coveropt{3}{H}\le \frac{59}{2}q^3 -24q^2d +\frac{13}{2}qd^2 -\frac{5}{9}d^3+Q_0,$$
where $Q_0=\max\{Q_i\}_{i=1}^3$
and   $Q_1=\frac{13}{2}q^2-\frac{17}{6}qd+\frac{2}{9}d^2+\frac{1}{3}q-\frac{1}{9}d+\frac{4}{9},$
       $Q_2=-23q^2+\frac{79}{6}qd-\frac{35}{18}d^2+\frac{11}{2}q-\frac{11}{6}d,$ and
       $Q_3= -\frac{105}{2}q^2+\frac{175}{6}qd-\frac{37}{9}d^2+\frac{92}{3}q-\frac{80}{9}d-\frac{16}{3}.$

Assume that $Q_0=Q_1$ (the other two cases are analogous, see Appendix C).
Then, we have
\begin{align}\label{Equ:CC3(H)<F(q,d)}
    \coveropt{3}{H} \leq	
    \frac{59}{2}q^3 -24q^2d +\frac{13}{2}qd^2 -\frac{5}{9}d^3
    +\frac{13}{2}q^2-\frac{17}{6}qd+\frac{2}{9}d^2+\frac{1}{3}q-\frac{1}{9}d+\frac{4}{9}:=F(q,d). 
\end{align}
The partial derivative $F'_d(q,d)=-24q^2+13qd-\frac 5 3 d^2-\frac {17} 6 q + \frac 4 9 d - \frac 1 9.$
By~\eqref{Equ:range of q}, $3q<d\le 4q.$
It is not hard to see that
$F'_d(q,3q)=-\frac 3 2 q - \frac 1 9 <0,$
$F'_d(q,3q+1)=\frac{3}2 q-\frac{4}{3}>0,$ and
$F'_d(q,4q)=\frac 4 3 q^2-\frac {19}{18} q - \frac 1 9>0$
holds for $q\geq 1.$
Therefore, $F(q,d)$ is increasing with respect to $d.$ 

If %$4q<n-q,$ i.e., 
$q<\frac n 5,$
then $$F(q,d)\le F(q,4q)=\frac{35}{18}q^3-\frac{23}{18}q^2-\frac{1}{9}q+\frac 4 9<
\frac{7}{450}n^3-\frac{23}{450}n^2+\frac{1}{45}n+\frac 4 9<\frac{1}{64}(n-1)^3,$$
where the second inequality holds since the function $F(q,4q)$ increases for $q\ge1,$
and the last inequality holds when $n\ge 8.$
Thus, $\coveropt{3}{H} \le F(q,d) < \frac{1}{64}(n-1)^3\le k_4(\turangraph{n}{4})-k_4(\turangraph{n-1}{4}),$ and we are done.

If %$4q\ge n-q,$ i.e., 
$q\ge \frac n 5,$ since $d\le n-q,$ we have
\begin{equation}\label{Equ:g(q)}
    F(q,d)\le F(q,n-q)=\frac{545}{9}q^3-\frac{116}{3}q^2n+\frac{49}{6}qn^2-\frac{5}{9}n^3+\frac{86}{9}q^2-\frac{59}{18}qn+\frac{2}{9}n^2+\frac{4}{9}q-\frac{1}{9}n+\frac 4 9.
\end{equation}
Let the result equation be $g(q).$ 
Then $g'(q)=\frac{545}{3}q^2-\frac{232}{3}qn+\frac{49}{6}n^2+\frac{172}{9}q-\frac{59}{18}n+\frac{4}{9}.$
Since $\frac n 5\le q\le n-\lfloor\frac{3}{4}n\rfloor-1=\lceil\frac{n}{4}\rceil-1<\frac n 4,$
it is easy to see that 
$g'(0)=\frac{49}{6}n^2-\frac{59}{18}n+\frac{4}{9}>0,$
$g'(\frac n 5)=-\frac{1}{30}n^2+\frac{49}{90}n+\frac{4}{9}<0,$ and
$g'(\frac n 4)=\frac{3}{16}n^2+\frac{2}{3}n+\frac{4}{9}>0,$
%$$g'(\frac n 5)=-\frac{1}{30}n^2+\frac{49}{90}n+\frac{4}{9}<0\ \  \text{and}\ \ 
%g'(\frac n 4)=\frac{3}{16}n^2+\frac{2}{3}n+\frac{4}{9}>0,$$ 
when $n\ge 18.$
Thus, $g(q)$ is decreasing first and then increasing with respect to $q$
and the maximum of $g(q)$ will take value at the endpoints of the interval containing $q,$ i.e., 
$g(q)\le \max\{g(\frac n 5),g(\lceil\frac{n}{4}\rceil-1)\}.$
Note that 
$g(\frac n 5)=F(\frac n 5,\frac {4n} 5)=\frac{7}{450}n^3-\frac{23}{450}n^2+\frac{1}{45}n+\frac 4 9<\frac{1}{64}(n-1)^3$ holds for $n\ge 8.$
And 
$g(\frac{n}{4}-2)=\frac{1}{64}n^3-\frac{3}{8}n^2+24n-\frac{1340}{3}<\frac{1}{64}(n-1)^3$ holds for all $n.$
Therefore, if $q\le \frac n 4-2,$ 
then $\coveropt{4}{G}\le \coveropt{4}{G'}+\coveropt{3}{H}<\coveropt{4}{\turangraph{n}{4}},$ 
and we are done.

Now, the only remaining case is $q=\lceil\frac{n}{4}\rceil-1.$
In this situation, we have 
$\delta(G)=d=n-q=\lfloor \frac{3}{4}n \rfloor +1.$
We will analyze more carefully depending on the residue of $n$ modulo $4.$ 
By~\eqref{Equ:g(q)}, we have
\begin{align*}
(i).\,  n=4k,\, q=k-1: \quad & g(q)=k^3-3k^2-\frac{329}{3}k-51 < k^3 = k_4(\turangraph{4k}{4}) - k_4(\turangraph{4k-1}{4}), \\
(ii).\, n=4k+1,\, q=k: \quad & g(q) = k^3 = k_4(\turangraph{4k+1}{4}) - k_4(\turangraph{4k}{4}), \\
(iii).\, n=4k+2,\, q=k: \quad & g(q) = k^3+3k-\frac{10}{3}<k^3+k^2 < k^3 + k^2 = k_4(\turangraph{4k+2}{4}) - k_4(\turangraph{4k+1}{4}), \\
(iv).\, n=4k+3,\, q=k: \quad & g(q) = k^3+9k-\frac{116}{9} < k^3 + 2k^2 + k = k_4(\turangraph{4k+3}{4}) - k_4(\turangraph{4k+2}{4}).
\end{align*}
%If $n=4k,$ then $q=k-1.$ 
%By~\eqref{Equ:g(q)},
%$g(q)=k^3-3k^2-\frac{329}{3}k-51<k^3=k_4(\turangraph{4k}{4})-k_4(\turangraph{4k-1}{4}).$
%If $n=4k+1$, then $q=k.$ 
%By~\eqref{Equ:g(q)},
%$g(q)=k^3=k_4(\turangraph{4k+1}{4})-k_4(\turangraph{4k}{4}).$
When $n=4k+1,q=k$, if all the equalities hold,
then by induction, $G'$ is isomorphic to $\turangraph{4k}{4}.$
Since $\delta(G)=d=3k+1,$
we conclude that $v$ is adjacent to all vertices of $G',$ 
contradicting that $v$ has the minimum degree.
%Similarly, if $n=4k+2$, then $q=k.$ 
%By~\eqref{Equ:g(q)},
%$g(q)=k^3+3k-\frac{10}{3}<k^3+k^2=k_4(\turangraph{4k+2}{4})-k_4(\turangraph{4k+1}{4}).$
%And if $n=4k+3$, 
%then $q=k,$ and from~\eqref{Equ:g(q)},
%$g(q)=k^3+9k-\frac{116}{9}<k^3+2k^2+k=k_4(\turangraph{4k+3}{4})-k_4(\turangraph{4k+2}{4}).$
Therefore, we have proved that when 
$\lfloor \frac{3}{4}n \rfloor +1\leq d \leq \frac{4}{5}n,$
$\coveropt{3}{H}< k_4(\turangraph{n}{4})-k_4(\turangraph{n-1}{4}).$
Combining all the cases with~\eqref{Equ:CC_4(G)<=CC_4(G')+CC_3(H)}, 
when $\omega(G)=5,$ we have 
$$\coveropt{4}{G}\le \coveropt{4}{G'}+\coveropt{3}{H}<k_4(\turangraph{n-1}{4})+k_4(\turangraph{n}{4})-k_4(\turangraph{n-1}{4})=\coveropt{4}{\turangraph{n}{4}},$$
%And when $\omega(G)=4,$ by the analysis above, from Lemma~\ref{Lem:Erdos},
%we have 
%$$\mathcal{C}\mathcal{C}_4(G)=k_4(G)\le h(n,5,4)=k_4(\turangraph{n}{4})=\coveropt{4}{\turangraph{n}{4}},$$
%equality holds if and only if $G=\turangraph{n}{4},$
which completes the proof of Theorem~\ref{Thm:Main}.
\QED

\section{Proof of Lemma~\ref{Lem:key lemma}}\label{Sec:Proof of key lemma}
In this section, we complete the proof of Lemma~\ref{Lem:key lemma}.
Our proof strategy proceeds in the following three steps. 
First, in Subsection~\ref{Sec:greedy sequence and a general bound}, we generalize Lov\'asz's method~\cite{L68} to construct a 3-clique cover of $H$ using its greedy partition $\mathcal{F}_H$, and reformulate the cover size using the size of its corresponding \textit{greedy sequence} (Definition~\ref{Dfn:greedy sequence}).
Next, in Subsection~\ref{Sec:Optimize the values of greedy sequences}, we analyze extremal properties of these sequences.
Finally, in Subsection~\ref{Sec:Proof of better upperbound lemma}, we examine the subgraph induced by vertices of 4-cliques in $\mathcal{F}_H$, using its combinatorial properties to refine our bound and establish Lemma~\ref{Lem:key lemma}.

\subsection{3-clique cover construction}\label{Sec:greedy sequence and a general bound}
We define \textit{greedy sequence} and its value as follows. 

\begin{dfn}\label{Dfn:greedy sequence}
An \textit{$(m,p)$-greedy sequence} $A = (a_1, \dots, a_p)$ satisfies that $a_i\in\{1,2,3,4\}$ for $i\in [p],$ $a_i\ge a_{i+1}$ for $i\in [p-1]$, and $\sum_{i=1}^p a_i=m.$ 
Let $a = |\{i\in [p] : a_i = 4\}|$, $b = |\{i\in [p] : a_i \geq 3\}|$, $c = |\{i\in [p] : a_i \geq 2\}|$.
%For an $(m,p)$-greedy sequence $A=\{a_i\}_{i=1}^{p},$
%assuming that the numbers of $4, 3, 2$ and $1's$ in $A$ are $a, b-a, c-b$ and $p-c$ respectively.
Then, the \textit{value} of $A$ is denoted by $f(A):=S_1(A)+S_2(A)+S_3(A),$ where $S_1(A)=b,$ $S_2(A)=mb+a^2-ab-b^2+bc-a-3b,$ and 
\begin{align*}
    S_3(A)=&\sum_{k=3}^{b} a_k \sum_{j=2}^{k-1} (a_j-1) (j-1) +\sum_{k=b+1}^{c} a_k
           \left( \sum_{j=2}^{b}(a_j-1)(j-1) + \sum_{j=b+1}^{k-1}(a_j-1)b \right)+ \\
           &\sum_{k=c+1}^{p} \left( \sum_{j=2}^{b}(a_j-1)(j-1) + \sum_{j=b+1}^{c}(a_j-1)b \right).
\end{align*}
\end{dfn}
\noindent
Since $H$ is $K_5$-free in Lemma~\ref{Lem:key lemma}, if $\dF_H=\{A_1,\cdots,A_p \}$ is a greedy partition of $H$ of size $p$,
then $A=\{|A_i|\}_{i=1}^{p}$ is an $(|H|,p)$-greedy sequence which we call the corresponding greedy sequence of $\dF_H.$

%Given a greedy partition $\dF_H$ of $H$, our next lemma gives an upper bound on $\coveropt{3}{H}$ by the value of the corresponding greedy sequences of $\dF_H$.

\begin{lem}\label{Lem:Bound of 3-clique cover by f(A)}
     If $H$ is a $K_5$-free graph and $\dF_H=\{A_1,\cdots,A_p \}$ is a greedy partition of $H$ of size $p$.  
     Let $A=\{|A_i|\}_{i=1}^{p}$ be the corresponding $(|H|,p)$-greedy sequence of $\dF_H$.
     Then, we have 
	$\coveropt{3}{H}\leq f(A).$
\end{lem}
\begin{proof}
We prove the lemma by constructing a 3-clique cover $\dC$ of $H$ of size at most $f(A).$

Let $a_i = |A_i|$, and let $a,b,c$ be as in Definition~\ref{Dfn:greedy sequence}. 
For $i\in[3],$ let $\dS_i$ be the triangles with three vertices lying in exactly $i$ different $A_j$'s.
First, we can easily see that  $\dC_1:=\bigcup_{i}\{A_i: a_i\geq 3\}$ is a set of cliques that covers all the triangles in $\dS_1.$
Thus, we have $|\dC_1|=|\cup_{i}\{A_i: a_i\geq 3\}|=b=S_1(A).$

For a clique $C$ in $H$ and a vertex set $U\subset V(H)$, recall that $T_U[C]$ is the induced subgraph of $H$ on vertex set $C\cup U',$ where $U'$ is the common neighbor of the vertices in $C$ in $U.$   
Let $\mathcal{I}_1=\{(i,j): 1\le i<j\le p\ \text{and}\ i\le b \},$
    $\mathcal{I}_2=\{(i,j): 1\le i<j\le a\},$ 
    $\mathcal{I}_3=\{(i,j): 1\le i<j\le p,\ i\le b\ \text{and}\ j>a\},$ and 
\begin{align*}
\dC_2=\bigcup\limits_{(i,j)\in\mathcal{I}_1} \{T_{A_i}[v]: v\in A_j\}
      \bigcup\limits_{(i,j)\in\mathcal{I}_2} \{T_{A_j}[v]: v\in A_i\}
      \bigcup\limits_{(i,j)\in\mathcal{I}_3} \{T_{A_i}[u,v]: uv \in E(A_j)\}.
\end{align*}
We claim that $\dC_2$ covers all the triangles in $\dS_2.$ 
Indeed, by %Proposition~\ref{Claim:p<n-delta} and 
the definition of $\dF(H)$, $H[A_{b+1}\cup\cdots\cup A_p]$ is triangle-free.
Thus, the first union set covers all the triangles with one vertex lying in $A_j$ and the other two lying in $A_i$ with $1\leq i<j\leq p$, and the size of the union set is 
\begin{align*}
&\sum_{j=1}^ba_j(j-1)+\sum_{j=b+1}^pa_jb=3\sum_{j=1}^b (j-1)+\sum_{j=1}^a (j-1)+b\sum_{j=b+1}^pa_j \nonumber \\
=&\frac{3}{2}b(b-1)+\frac{1}{2}a(a-1)+b(m-a-3b)=
mb+\frac{a^2}{2}-ab-\frac{3}{2}b^2-\frac {a}{2}-\frac{3}{2}b.
\end{align*} 
Note that the remaining triangles in $\dS_2$ has one vertex lying in $A_i$ and the other two lying in $A_j$ with some $1\le i<j\le p$.
If $|A_i|<|E(A_j)|,$ then we use the second union set to cover them,
and the size of the union set is
$\sum_{i=1}^{a}a_i(a-i)=4\sum_{i=1}^{a}(a-i)=2a^2-2a.$
If $|A_i|\ge |E(A_j)|,$ then we use the third union set to cover them,
and the size of the union set is
\begin{align*}
\sum_{j=a+1}^b\binom{a_j}{2}(j-1)+\sum_{j=b+1}^c\binom{a_j}{2}b
=3\sum_{j=a+1}^b (j-1)+\sum_{j=b+1}^c b=-\frac{3}{2}a^2+ \frac{1}{2}b^2 +bc +\frac 3 2a-\frac 3 2 b.
\end{align*}
Thus, $\dC_2$ is a set of cliques that covers all the triangles in $\dS_2$ (note that $\dC_2$ may contain vertices or edges).
Combining the above three equalities, we get that 
\begin{align*}
    |\dC_2|= mb+a^2-ab-b^2+bc-a-3b=S_2(A).
\end{align*}
    
In a similar way, to cover all the triangles in $\dS_3,$ we define
\begin{align*}
\dC_3=\bigcup\limits_{i,j,k} \{T_{A_i}[u,v]: u\in A_j, v\in A_k, uv\in E(H),\ 1\le i< j<k\le p\}.
\end{align*}
Note that $\dC_3$ may contain edges.
By Proposition~\ref{Claim:p<n-delta}, 
for any $j<k$ and $v\in A_k$, $v$ has at most $a_j-1$ neighbors in $A_j.$   
By the definition of $\dF(H),$ $H[A_{b+1}\cup\cdots\cup A_p]$ is triangle-free and $H[A_{c+1}\cup\cdots\cup A_p]$ is empty, according to the range of $i,j,k$, 
we have
\begin{align}\label{Equ:upper bound on C3}
|\dC_3|\leq &\sum_{k=3}^{b} a_k \sum_{j=2}^{k-1} (a_j-1) (j-1)
+\sum_{k=b+1}^{c} a_k
\left( \sum_{j=2}^{b}(a_j-1)(j-1) + \sum_{j=b+1}^{k-1}(a_j-1)b \right)\nonumber \\
+&\sum_{k=c+1}^{p} \left(  \sum_{j=2}^{b}(a_j-1)(j-1) + \sum_{j=b+1}^{c}(a_j-1)b \right)=S_3(A).
\end{align}

Let $\dC=\dC_1\cup \dC_2 \cup \dC_3\backslash D,$ where $D=\{V(H)\cup E(H)\}.$
Then, by the analysis above and the equations 
$|\dC_1|=S_1(A)$, $|\dC_2|=S_2(A),$ and $|\dC_3|\le S_3(A)$, 
we get that $\dC$ is a $3$-clique cover of $H$ of size at most $f(A)$, which completes the proof of Lemma~\ref{Lem:Bound of 3-clique cover by f(A)}. 
\end{proof}

\subsection{Greedy sequence analysis}\label{Sec:Optimize the values of greedy sequences}
In this subsection, for the $(m,p)$-greedy sequence $A=\{a_i\}_{i=1}^p$, we will use the adjustment operation in the following lemma to preserve the sum $m$ while increasing $f(A)$. 
Indeed, we can determine the structure of the greedy sequence that achieves the maximum value.

\begin{lem}\label{Lem:adjustment method}
	For any $(m,p)$-greedy sequence $A$, there exists an integer $q\le p$ and an $(m,q)$-greedy sequence $A'$ such that 
    $$f(A)\le f(A')\le 28q^3-\frac{45}{2}q^2m+6qm^2-\frac{1}{2}m^3+R_0,$$
    where $R_0=\max\{R_i\}_{i=1}^3$ with 
    $R_1=\frac{3}{2}qm-\frac{1}{2}m^2-3q+m$,
    $R_2=-28q^2+\frac{33}{2}qm-\frac{5}{2}m^2+6q-2m$ and 
    $R_3=-56q^2+\frac{63}{2}qm-\frac{9}{2}m^2+34q-10m-6.$ 
\end{lem}

We define three operations for a greedy sequence $A$
(applied in priority order Operation 1$>$Operation 2$>$Operation 3).
Note that the number of 4's remains unchanged throughout all operations.
\begin{itemize}[leftmargin=22mm]
    \setlength\itemsep{0em}
	\item [Operation 1.] 
      If there are at least two 1's. 
      Suppose that $i$ and $j$ are the minimum and maximum indices such that $a_i=a_j=1.$ 
	 In this case, we set $a_i \to 2$, $a_j \to 0$ and $a_k \to a_k$ 
      for all $k\notin\{i,j\}$. 
      Then remove $a_j$ from the new sequence, reducing the length by 1.
	\item [Operation 2.] 
      If there is exactly one 1 and at least one 2. 
      Suppose that $i$ is the minimum index such that $a_i=2$ and $j$ is the index such that $a_j=1.$
      We set $a_i \to 3$, $a_j \to 0$ and $a_k \to a_k$ 
      for all $k\notin\{i,j\}$.  
      Then remove $a_j$ from the new sequence, reducing the length by 1.
	\item [Operation 3.] 
      If there are at least two 2's and no 1. 
      Suppose that $i$ and $j$ are the minimum and maximum indices such that $a_i=a_j=2.$ 
	 Then we set $a_i \to 3,$ $a_j \to 1,$ and $a_k \to a_k$ for all 
      $k\notin\{i,j\}.$
\end{itemize}
Repeated application of Operations 1-3 (applied in priority order) transforms an $(m,p)$-greedy sequence $A$ into an $(m,q)$-greedy sequence $A' = (a_1', \dots, a_q')$ with $q \leq p$ that is \textit{irreducible} (no operations can be applied anymore) and belongs to exactly one of the following three types. 
\begin{itemize}[leftmargin=22mm]
\setlength\itemsep{0em}
\item[\textbf{Type 1:}] $a_i' \in \{3,4\}$ for all $i \in [q].$ 
\item[\textbf{Type 2:}] $a_i' \in \{3,4\}$ for $i \in [q-1]$, and $a_q' = 2.$
\item[\textbf{Type 3:}] $a_i' \in \{3,4\}$ for $i \in [q-1]$, and $a_q' = 1.$
\end{itemize}

First, we prove that for any greedy sequence $A$,
during the process, if we perform the three operations on $A$ in order of priority, the value $f$ of $A$ will not decrease (Claims~\ref{Clm:adjust 1}-\ref{Clm:adjust 3}).
Therefore, we only need to compute the values of the above three types of greedy sequence. 

Let $A_{\text{old}} = (a_1,\dots,a_p)$ be an $(m,p)$-greedy sequence with $a \times 4$'s, $(b-a) \times 3$'s, $(c-b) \times 2$'s, and $(p-c) \times 1$'s. 
We also use the notation $f(4_1,...4_a,3_{a+1},...,3_b,2_{b+1},...,2_c,1_{c+1},..,1_p)$ to represent $f(A_{\text{old}}).$
For any operation, define $\Delta f =f(A_{\text{new}}) - f(A_{\text{old}}) = \Delta S_1 + \Delta S_2 + \Delta S_3,$
where $\Delta S_i = S_i(A_{\text{new}}) - S_i(A_{\text{old}})$ and $A_{\text{new}}$ is the resulting sequence after performing the operation on $A_{\text{old}}.$
\begin{clm}\label{Clm:adjust 1}
	If $p\geq c+2$, then perform Operation 1 on $A_{\text{old}},$ 
    we will get that $\Delta f\ge0.$
    %and so Operation 1 will not decrease the value of f.
\end{clm}
\begin{proof}
It is easy to see that $A_{\text{new}}:=(a_1',...,a_{p-1}')=(a_1,\cdots, a_c,2_{c+1},1_{c+2},\cdots,1_{p-1})$ and 
by Definition~\ref{Dfn:greedy sequence}, we have 
$\Delta S_1 = b-b = 0,$ 
$\Delta S_2 = b(c+1)-bc = b\ge 0,$
and
\begin{align*}
\Delta S_3
&=
2\sum_{j=b+1}^{c}(a_j-1)b + \sum_{k=c+2}^{p-1}\sum_{j=b+1}^{c+1}(a_j'-1)b  - \sum_{k=c+1}^{p}\sum_{j=b+1}^{c}(a_j-1)b\\
&=
2\sum_{j=b+1}^{c}b + \sum_{k=c+2}^{p-1}\sum_{j=b+1}^{c+1}b - \sum_{k=c+1}^{p} \sum_{j=b+1}^{c}b
=b(p-c-2) \geq 0,
\end{align*}
and we are done.
\end{proof}

After repeatedly using Operation 1, we can assume that there is at most one 1 in $A_{\text{old}}=\{a_1,...,a_p\}$.

\begin{clm}\label{Clm:adjust 2}
    If $p=c+1,$ 
    then perform Operation 2 on $A_{\text{old}},$ 
    we will get that $\Delta f\ge0.$
\end{clm}
\begin{proof}
Now, we have $A_{\text{old}}=(a_1,\cdots, a_b,2_{b+1},\cdots,2_{p-1},1_{p})$ and 
$A_{\text{new}}:=(a_1',...,a_{p-1}')=(a_1,\cdots, a_b,\\ 3_{b+1},2_{b+2},\cdots,2_{p-1}).$
Since $m=p+a+b+c$ and $p=c+1,$ by Definition~\ref{Dfn:greedy sequence}, 
we get
$\Delta S_1 =(b+1)-b=1,$ 
$\Delta S_2=m(b+1)-a(b+1)-(b+1)^2+(b+1)c-3(b+1)-(mb-ab-b^2+bc-3b)
=m-a-2b-1+c-3=-b+3c-3,$
and we have $\Delta S_3$ is equal to
\begin{align*}
&3\sum_{j=2}^{b}(a_j-1)(j-1)
+2\left(\sum_{k=b+2}^{c}(k-b-2)(b+1)-\sum_{k=b+1}^{c}(k-b-1)b\right) + 2\sum_{k=b+2}^{c}
\Bigg(\sum_{j=2}^{b+1}(a_j'-1)(j-1)\\
&-\sum_{j=2}^{b}(a_j-1)(j-1)\Bigg) - 2\sum_{j=2}^{b}(a_j-1)(j-1) +\sum_{k=c+1}^{p-1}(c-1)-\sum_{j=b+1}^{c}b-\sum_{j=2}^{b}(a_j-1)(j-1) \\
=&2\sum_{k=b+2}^{c}(k-2b-2)+2\sum_{k=b+2}^{c}(2b)-b(c-b)
=-bc+c^2+b-3c+2.
\end{align*}
Thus, we get that
$\Delta f= \Delta S_1+\Delta S_2+\Delta S_3= -bc+c^2 =c(c-b) \geq 0,$ 
and we are done.
\end{proof}

After using Operation 2, we can assume that there is no $1$ in $A_{\text{old}}=\{a_1,...,a_p\}$.

\begin{clm}\label{Clm:adjust 3}
	If $p=c\geq b+2$, 
    then perform Operation 3 on $A_{\text{old}},$ 
    we will get that $\Delta f\ge0.$
\end{clm}
\begin{proof}
Now, we have $A_{\text{old}}=(a_1,\cdots, a_b,2_{b+1},\cdots,2_p)$ and 
$A_{\text{new}}:=(a_1',...,a_{p}')=(a_1,\cdots, a_b,3_{b+1},2_{b+2},\\\cdots,2_{p-1},1_p).$
Since $m=p+a+b+c$ and $p=c,$ by Definition~\ref{Dfn:greedy sequence}, 
we get that
$\Delta S_1 =(b+1)-b=1,$ 
$\Delta S_2=m-a-2b-1-b+c-1-3=-2b+3c-5,$
and $\Delta S_3$ is equal to
\begin{align*}
&3\sum_{j=2}^{b}(a_j-1)(j-1)
+2\sum_{k=b+2}^{c-1}(k-2)
-2 \sum_{j=b+1}^{c-1}b - 4\sum_{j=2}^{b}(a_j-1)(j-1) + \sum_{j=b+2}^{c-1}(b+1) + \sum_{j=2}^{b+1}(a_j'-1)\cdot \\
& (j-1)
=2\sum_{k=b+2}^{c-1}(k-2)-2b(c-b-1)+(b+1)(c-b-2)+2b
=c^2-bc+2b-4c+4.
\end{align*}
Thus, we get that $\Delta f =\Delta S_1+\Delta S_2+\Delta S_3= -bc+c^2-c=c(c-b-1)\geq 0,$ and we are done.
\end{proof}

Now, we are ready to complete the proof of Lemma~\ref{Lem:adjustment method}.

\noindent
\textbf{Proof of Lemma~\ref{Lem:adjustment method}.}
We consider any $(m,p)$-greedy sequence $A$. 
Through successive applications of Operations 1-3 with priority Operation 1$>$Operation 2$>$Operation 3, we obtain an irreducible $(m,q)$-greedy sequence $A' = (a_1',\ldots,a_q')$ with $q \leq p$ that belongs to one of Types 1-3. 

By Claims~\ref{Clm:adjust 1}, \ref{Clm:adjust 2}, and \ref{Clm:adjust 3}, each operation does not decrease the value of $f$, which implies that $f(A) \leq f(A').$
If $A'$ is of Type 1, 
then $A'=(4,...4,3,...,3).$
We can see that there are $(m-3q)$ many $4$'s and $(4q-m)$ many $3$'s.
To calculate $f(A'),$
by Definition~\ref{Dfn:greedy sequence}, 
plugging into 
$a=m-3q,$ $b=c=q$ and $p=q,$ 
we get that
$S_1 (A') = q$, 
$S_2(A') = (m-3q)^2+3q^2-m=12q^2-6qm+m^2-m$
and
\begin{align*}
S_3(A')=& \sum_{k=3}^{m-3q} 4\times 3\sum_{j=2}^{k-1} (j-1)
+\sum_{k=m-3q+2}^{q} 3 \times 2\sum_{j=m-3q+1}^{k-1} (j-1)
+  \sum_{k=m-3q+1}^{q} 3 \times 3 \sum_{j=2}^{m-3q} (j-1)\\ 
=&12\binom{m-3q}{3}+(4q-m)(4q-m-1)(2m-5q-2) + \frac 9 2(m-3q)(m-3q-1)(4q-m)\\
=&28q^3 - \frac{45}{2}q^2m +6qm^2-\frac{1}{2}m^3
-12q^2+\frac{15}{2}qm -\frac{3}{2}m^2
-4q+2m.  
\end{align*}
Thus, we have 
$ f(A')
=S_1(A')+S_2(A')+S_3(A')
=28q^3 - \frac{45}{2}q^2m +6qm^2-\frac{1}{2}m^3 + R_1,$
where $R_1=\frac{3}{2}qm -\frac{1}{2}m^2-3q+m.$
The cases for Type 2 and Type 3 sequences follow similarly (see detailed calculations in Appendix D).
Indeed, if $A'$ is of Type 2, 
we will get $ f(A')=28q^3 - \frac{45}{2}q^2m +6qm^2-\frac{1}{2}m^3 + R_2,$
where $R_2=-28q^2+\frac{33}{2}qm-\frac{5}{2}m^2+6q-2m,$
and if $A'$ is of Type 3, 
we will get $ f(A')=28q^3 - \frac{45}{2}q^2m +6qm^2-\frac{1}{2}m^3 + R_3,$
where $R_3=-56q^2+\frac{63}{2}qm-\frac{9}{2}m^2+34q-10m-6.$
Therefore, for any $(m,p)$-greedy sequence $A$, there exists an integer $q\le p$ and an $(m,q)$-greedy sequence $A'$ such that $f(A)\le f(A')\le 28q^3-\frac{45}{2}q^2m+6qm^2-\frac{1}{2}m^3+R_0,$
where $R_0=\max\{R_1,R_2,R_3\},$ 
which completes the proof of Lemma~\ref{Lem:adjustment method}.
\QED

\subsection{Structural refinement: completing the proof of Lemma~\ref{Lem:key lemma}}\label{Sec:Proof of better upperbound lemma}
To complete the proof of Lemma~\ref{Lem:key lemma}, we refine the bound from Lemma~\ref{Lem:Bound of 3-clique cover by f(A)}. 
Recall that for a graph $H$ with greedy partition $\mathcal{F}_H = \{A_1,\dots,A_p\}$, 
we constructed a $3$-clique cover of $H$ of size at most $S_1(A)+S_2(A)+S_3(A),$ 
where $S_3(A)$ bounds the number of cliques needed to cover all the triangles with three vertices lying in three different $A_i$'s, i.e., the upper bound of $\dC_3$, where  
\begin{align*}
\dC_3=\bigcup\limits_{i,j,k} \{T_{A_i}[u,v]: u\in A_j, v\in A_k, uv\in E(H),\ 1\le i< j<k\le p\}.
\end{align*}
Let $a$ be the number of $4$-cliques in $A_i$'s. 
Recall that in inequality~\eqref{Equ:upper bound on C3},
by Proposition~\ref{Claim:p<n-delta},
we use $\sum_{k=3}^{a} a_k \sum_{j=2}^{k-1} (a_j-1) (j-1)$ cliques
to cover all the triangles with three vertices lying in $A_i, A_j,A_k$ for all $1\leq i<j<k\le a.$
More precisely, for any fixed triple $(i,j,k)$ with $1\leq i<j<k\le a$, 
we use $4\times3=12$ cliques to cover all the triangles with three vertices lying in $A_i, A_j$ and $A_k,$ respectively. 

In fact, we can significantly improve the bound $\sum_{k=3}^{a} a_k \sum_{j=2}^{k-1} (a_j-1)(j-1)$ by developing a more efficient covering scheme for triangles with three vertices lying in $A_i, A_j,A_k$ for all $1\leq i<j<k\le a$. 
This refinement of $f(A)$ in Lemma~\ref{Lem:Bound of 3-clique cover by f(A)}, combining with Lemma~\ref{Lem:adjustment method}, will complete the proof of Lemma~\ref{Lem:key lemma}.
The key improvement comes from analyzing the subgraph $H[\bigcup_{i=1}^a A_i]$. 
We construct an auxiliary 3-uniform hypergraph $\mathcal{H} = (V(\mathcal{H}), E(\mathcal{H}))$, where
$$ V(\mathcal{H}):=\{ v_i: 1\leq i\leq a\}\ 
\text{and}\ E(\mathcal{H}):= \{ v_iv_jv_k: 1\le i<j<k\le a\ \text{and}\ H[A_i\cup A_j\cup A_k]= \turangraph{12}{4} \}.$$
Since $H$ is $K_5$-free, 
$e(H[A_i\cup A_j\cup A_k])\le e(\turangraph{12}{4})=54,$ 
equality holds if and only if $H[A_i\cup A_j\cup A_k]=\turangraph{12}{4}.$
The following claim states that improvement can be made if $e(H[A_i\cup A_j\cup A_k])<e(\turangraph{12}{4}).$
\begin{clm}\label{Clm:Use less cliques}
	If $1\leq i<j<k\le a$ and
    $e(H[A_i\cup A_j\cup A_k])< e(\turangraph{12}{4}),$
    then at most 11 cliques are needed to cover all the triangles with three vertices lying in $A_i, A_j$ and $A_k,$ respectively.   
\end{clm}
\begin{proof}
If $e(H[A_i\cup A_j\cup A_k])< e(\turangraph{12}{4}),$
then by the pigeonhole principle, at least one of the number of 
$\{e(A_i,A_j),e(A_i,A_k),e(A_j,A_k)\}$ is at most $11.$ 
Without loss of generality, assume that $e(A_i,A_j)\le 11.$ 
Then we can use the set of cliques
$\{T_{A_k}[u,v]: u\in A_i, v\in A_j, uv\in E(H)\}$
to cover all the triangles with three vertices lying in $A_i, A_j, A_k,$ respectively,
which is of size at most 11.
\end{proof}

Let $|E(\dH)|=\dE$ and $|E(\overline{\dH})|=\overline{\dE}$, where $\overline{\dH}$ denotes the complement graph of $\dH.$
The previous claim establishes that for each hyperedge in $\overline{\mathcal{H}}$, we can reduce our bound on $\coveropt{3}{H}$ by 1 compared to the estimation in Lemma~\ref{Lem:Bound of 3-clique cover by f(A)}. 
This yields an improvement when $\dE$ is small.
Conversely, when $\dE$ is large (meaning many triples induce Tur\'an graphs), we employ a different optimization strategy based on the following structural claim about our auxiliary hypergraph $\mathcal{H}$. 
Let $K_4^{(3)}$ denote the complete 3-uniform hypergraph on 4 vertices, and $K_4^{(3)-}$ denote $K_4^{(3)}$ minus one edge.
\begin{clm}\label{Clm:K4(3)-}
	If $\dH$ contains a $K_{4}^{(3)-}$ on $\{v_i,v_j,v_k,v_\ell\}$, then it must be a $K_{4}^{(3)}.$
    Moreover, we can use 24 4-cliques to cover all the triangles with three vertices lying in three different cliques of $A_i, A_j, A_k, A_\ell$.
\end{clm}
\begin{proof}
	Without loss of generality, assume that the three hyperedges of the $K_{4}^{(3)-}$ on $\{v_i,v_j,v_k,v_\ell\}$ are $v_iv_jv_k, v_iv_jv_\ell$ and $v_iv_kv_\ell.$
	We claim that the hyperedge $v_jv_kv_\ell$ must exist.
	Indeed, the hyperedges $v_iv_jv_k, v_iv_jv_\ell$ and $v_iv_kv_\ell$ exist
	means that all of the three pairs $(A_j,A_k),$ $(A_j,A_\ell)$ and $(A_k,A_\ell)$ of vertices in $H$ induce a $\turangraph{8}{4}.$
	Since $H$ is $K_5$-free, we can see that $(A_j,A_k,A_\ell)$ must induce a $\turangraph{12}{4}.$
	Thus, the hyperedge $v_jv_kv_\ell$ exists and $\dH[\{v_i,v_j,v_k,v_\ell\}]$ is a $K_{4}^{(3)}.$

    For the second statement, since $\dH[\{v_i,v_j,v_k,v_\ell\}]$ is a $K_{4}^{(3)},$ 
    we can see that $H[A_i\cup A_j\cup A_k\cup A_\ell]$ is a $\turangraph{16}{4}.$  
    Then we can easily use $4\times3\times2\times1=24$ $4$-cliques to cover all the triangles with three vertices lying in three different cliques of $A_i, A_j, A_k, A_\ell$, and we are done.
\end{proof}

Note that in Lemma~\ref{Lem:Bound of 3-clique cover by f(A)}, our initial bound for covering these triangles was $4 \times 3 \times \binom{4}{3} = 48$ cliques. 
Claim~\ref{Clm:K4(3)-} shows this can be reduced to 24 many $4$-cliques when the hypergraph $\mathcal{H}$ contains complete $K_4^{(3)}$ subgraphs. 
This provides a crucial improvement when $\mathcal{H}$ is dense (i.e., when $\dE$ is large), as we will demonstrate by counting the number of $K_4^{(3)}$ subgraphs in $\mathcal{H}$.

With these preparations, we now complete the proof of Lemma~\ref{Lem:key lemma}.
\medskip

\noindent\textbf{Proof of Lemma~\ref{Lem:key lemma}.}
Throughout the proof of Lemma~\ref{Lem:adjustment method}, the number of 4's in the greedy sequence remains unchanged under all operations. 
Therefore, any improvement to $f(A)$ via Claims~\ref{Clm:Use less cliques} and~\ref{Clm:K4(3)-} directly translates to a corresponding improvement in $f(A')$.
%Consider two cases based on the density of $\mathcal{H}.$

If $\dE \leq \frac{2}{3} \binom{a}{3},$ 
i.e., $\overline{\dE}\ge \frac{1}{3} \binom{a}{3}.$
By Lemma~\ref{Lem:Bound of 3-clique cover by f(A)}, Lemma~\ref{Lem:adjustment method} and  Claim~\ref{Clm:Use less cliques},
we deduce that
\begin{align}\label{Equ:improvement when E is small}
\coveropt{3}{H}
&\leq f(A)- \overline{\dE}\leq f(A')-\frac{1}{3}\binom{a}{3}
\leq
28q^3-\frac{45}{2}q^2m+6qm^2-\frac{1}{2}m^3+R_0
-\frac 1 3 \binom{a}{3}\nonumber \\
&\leq 
28q^3-\frac{45}{2}q^2m+6qm^2-\frac{1}{2}m^3+R_0
-\frac{2}{9}(a-1)^2(a-2)+\frac{1}{6}a^2(a-1),
\end{align}
the last inequality holds since 
$\frac{2}{9}(a-1)^2(a-2)-\frac{1}{6}a^2(a-1)\le \frac 1 3 \binom{a}{3}$ for all $a\ge 0,$
where $a$ denotes the number of $4$'s in the $(m,p)$-greedy sequence $A=\{a_i\}_{i=1}^p,$
which is equivalent to the number of $4$'s in the $(m,q)$-greedy sequence $A'=\{a_i'\}_{i=1}^q$ we get from Lemma~\ref{Lem:adjustment method}.

If $A'$ is of Type 1,
then $a=m-3q,$ and $R_0=R_1=\frac{3}{2}qm-\frac{1}{2}m^2-3q+m.$
Thus, by~\eqref{Equ:improvement when E is small}, we get
\begin{align*}
\coveropt{3}{H}
&\leq
28q^3-\frac{45}{2}q^2m+6qm^2-\frac{1}{2}m^3
+\frac{3}{2}qm-\frac{1}{2}m^2-3q+m-\frac{2}{9}(m-3q-1)^2(m-3q-2)\\
&\ \ \ \ +\frac{1}{6}(m-3q)^2(m-3q-1)
=\frac{59} 2 q^3 - 24 q^2 m + \frac{13}{2} q m^2 - \frac{5}{9} m^3
   +Q_1,
\end{align*}
where $Q_1= \frac {13} 2 q^2 -\frac {17} 6 q m + \frac 2 9 m^2 + \frac{1}{3} q - \frac 1 9 m + \frac 4 9.$
When $A'$ is of Type 2 (Type 3), 
we can get the corresponding bounds are $\frac{59} 2 q^3 - 24 q^2 m + \frac{13}{2} q m^2 - \frac{5}{9} m^3+Q_2\ (Q_3),$ respectively (see the justification in Appendix E).
Therefore, we get 
\begin{align}\label{Equ:best bound of CC3(H)}
    \coveropt{3}{H}\le \frac{59} 2 q^3 - 24 q^2 m + \frac{13}{2} q m^2 - \frac{5}{9} m^3+Q_0,
\end{align}
where $Q_0=\max\{Q_1,Q_2,Q_3\},$ and we are done.

When $a\le 2,$ by the definition of $\dH$, $|V(\dH)|=a\le2,$ and 
$\dE=\overline{\dE}=0$.
By~\eqref{Equ:improvement when E is small}and \eqref{Equ:best bound of CC3(H)}, we are done.
Thus, in the following of the proof, we may assume that $a\ge 3.$

If $\dE \ge \frac{2}{3} \binom{a}{3},$
then we will show that $\mathcal{H}$ contains numerous $K_4^{(3)-}.$
By Claim \ref{Clm:K4(3)-}, this implies that $\mathcal{H}$ contains numerous $K_4^{(3)}.$ 
Each $K_4^{(3)}$ allows us to replace the original 48-clique covering of its corresponding $\turangraph{16}{4}$ with a more efficient 24-clique covering, which significantly reduces our bound on $\coveropt{3}{H}$.
However, note that the $3$-clique cover number obtained by this reduction will repeat if two $K_4^{(3)}$ intersect at more than two vertices in $\dH.$
If two $K_4^{(3)}$ intersect at most two vertices in $\dH,$
then there is nothing to worry about.

For $v\in V(\mathcal{H}),$ let $d_v$ be the degree of $v$, i.e., the number of hyperedges containing $v.$ 
The link graph of $v$ is the simple graph $H_v,$
with vertex set $V(H_v)=V(\mathcal{H}),$
and edge set 
$E(H_v)=\{xy:x,y\in V(\dH)\ \text{and}\ vxy\in E(\mathcal{H})\}.$
Let $k_3^v$ denote the number of triangles in the link graph of $v.$
By double counting the edges in $H_v$, we have
\begin{equation}\label{Equ:double counting edges}
\sum_{v\in V(\mathcal{H})} d_v = 3  \dE \geq 2 \binom{a}{3}.
\end{equation}
By Claim~\ref{Clm:K4(3)-}, 
one can easily check that a triangle on vertex set $\{x,y,z\}$ in $H_v$ with $v$ form a $K_4^{(3)}$ in $\dH.$
By Lemma~\ref{Lem:MM}, 
we have
$k_3^v\geq \frac{4}{3} \frac{d_v^2}{a} -\frac{a}{3} d_v,$ i.e., 
the number of $K_4^{(3)}$ containing $v$ is at least $\frac{4}{3} \frac{d_v^2}{a} -\frac{a}{3}d_v.$
Then, by Jensen's inequality and~\eqref{Equ:double counting edges}, the total number of  $K_4^{(3)}$ in $\dH$ is at least
\begin{align*}
&\frac{1}{4}\left(\frac{4}{3} 
\left( \frac{\sum_{v \in V(\mathcal{H})} d_v}{a} \right)^2
- \frac a 3\sum_{v \in V(\mathcal{H})} d_v \right)
=\frac 13 \left( \frac{3\dE}{a} \right)^2 - \frac 1 4 a \dE\\
\ge& \frac 1 3\left(\frac{2 \binom{a}{3}}{a} \right)^2
- \frac{1}{6} a \binom{a}{3}
=\frac{1}{27}(a-1)^2(a-2)^2-\frac{1}{36}a^2(a-1)(a-2),
\end{align*}
where the inequality holds since the function $\frac 13 \left( \frac{3\dE}{a} \right)^2 - \frac 1 4 a \dE$ of $\dE$ is increasing when $\dE \geq \frac{a^3}{24}$
which holds by the assumption that $\dE \ge \frac{2}{3} \binom{a}{3}$ and $a\ge 3.$

Note that for each $K_4^{(3)},$ there are at most $\binom 4 3(a-4)+1\leq 4(a-2)$ many
$K_4^{(3)}$ intersect more than $2$ vertices with it.
Thus, there are at least
$\frac{1}{108}(a-1)^2(a-2)-\frac{1}{144}a^2(a-1)$
$K_4^{(3)}$ in $\dH$ such that no two of them intersect with more than 2 vertices.
For each such $K_4^{(3)}$, Claim~\ref{Clm:K4(3)-} allows us to reduce the bound on $\coveropt{3}{H}$ in Lemma~\ref{Lem:Bound of 3-clique cover by f(A)} and~\ref{Lem:adjustment method} by 24 cliques without overlap. 
Thus, we get that
\begin{align*}
\coveropt{3}{H}
&\leq  
f(A)-\frac{2}{9}(a-1)^2(a-2)+\frac{1}{6}a^2(a-1)\le 
f(A')-\frac{2}{9}(a-1)^2(a-2)+\frac{1}{6}a^2(a-1)\\
&\leq
28q^3-\frac{45}{2}q^2m+6qm^2-\frac{1}{2}m^3+R_0
-\frac{2}{9}(a-1)^2(a-2)+\frac{1}{6}a^2(a-1),
\end{align*}
which is the same as the inequality~\eqref{Equ:improvement when E is small}.
Therefore, we obtain
$\coveropt{3}{H}\le \frac{59} 2 q^3 - 24 q^2 m + \frac{13}{2} q m^2 + \frac{5}{9} m^3+Q_0,$ 
with $Q_0=\max\{Q_1,Q_2,Q_3\}$ again, which completes the proof of Lemma~\ref{Lem:key lemma}.
\QED

\section*{Acknowledgement}
Tianying Xie was supported by National Key R\&D Program of China 2023YFA1010201 and National Natural Science Foundation of China grant 12471336. Jialin He was supported by Hong Kong RGC grant GRF 16308219 and Hong Kong RGC grant ECS 26304920.

\bibliographystyle{unsrt}

\begin{thebibliography}{99}
    \bibitem{BHKNW}
    J. Balogh, J. He, R. A. Krueger, T. Nguyen and M. C. Wigal,
    \newblock{Clique covers and decompositions of cliques of graphs},
    \newblock{arXiv:2412.05522}.

    
	\bibitem{BBC}
	N. Bansal, A. Blum and S. Chawla,
	\newblock{Correlation clustering},
	\newblock{\emph{Mach. Learn.}} \textbf{56} %(1-3)
	 (2004), 89--113.
	
	
	\bibitem{DMP}
	H. Dau, O. Milenkovic and G. J. Puleo,
	\newblock{On the triangle clique cover and $K_t$ clique cover problems},
	\newblock{\emph{Discrete Math.}} \textbf{343}, %(1)
	 2020.
	
	
	
	\bibitem{E62}
	P. Erd\H{o}s,
	\newblock{On the number of complete subgraphs contained in certain graphs},
	\newblock{\emph{Magy. Tud. Akad. Mat. Kut. Int\'{e}z. K\"{o}zl.}} \textbf{7} %(7) 
	(1962), 459--474.
	
	
	\bibitem{EGP}
	P. Erd\H{o}s, A. W. Goodman and L. P\'{o}sa,
	\newblock{The representation of a graph by set intersections},
	\newblock{\emph{Canad. J. Math.}} \textbf{18} % (1)
	(1966), 106--112. 
	
	
	\bibitem{LLM}
	J. Leskovec, K. J. Lang and M. W. Mahoney,
	\newblock{Empirical comparison of algorithms
	for network community detection},
    \newblock{\emph{in: Proceedings of the 19th International Conference on World Wide Web, in: WWW '10}}, 2010, pp. 631--640.
	
	
	\bibitem{L68}
	L. Lov\'{a}sz,
	\newblock{On covering of graphs},
	\newblock{\emph{in: Theory of Graphs}} (Proc. Colloq., Tihany, 1966),
	Academic Press, New York, 1968, pp. 231--236.

	
	\bibitem{MM}
	J. W. Moon and L. Moser,
	\newblock{On a problem of Tur\'{a}n},
	\newblock{\emph{Magy. Tud. Akad. Mat. Kut. Int\'{e}z. K\"{o}zl.}} \textbf{7} (1962), 283--286.

    
	\bibitem{PDFV}
	G. Palla, I. Der\'{e}nyi, I. Farkas and T. Vicsek,
	\newblock{Uncovering the overlapping community structure
	of complex networks in nature and society},
    \newblock{\emph{Nature}} \textbf{435} % (7043)
    (2005), 814--818. 


    \bibitem{R85} 
    F. S. Roberts, 
    \newblock{Applications of edge coverings by cliques},
    \newblock{\emph{Discrete Appl. Maths.}}, \textbf{10} (1) (1985), 93--109.

 

    \bibitem{ST99} 
    E. R. Scheinerman and A. N. Trenk, 
    \newblock{On the fractional intersection number}, 
    \newblock{\emph{Graphs Combin.}}, \textbf{15} (3) (1999), 341--351.


    \bibitem{Turan}
    P. Tur\'an,
    \newblock{On an extremal problem in graph theory},
    \newblock{\emph{Matematikai \'es Fizikai Lapok}}, \textbf{48} (1941), 436--452.
        
\end{thebibliography}

\subsection*{Appendix A: Justification of the inequality~\eqref{Equ:bound of clique number} for $6\le c \le 10$}\label{Sec:Appendix A}

Let $g(n,c)=\frac{(n-2)^2(n+2)^2}{256}-\left(\frac{n-c}{4}\right)^4 +1 +(n-c)+\binom c2\left(\frac{n-c}{c}\right)^2 +\binom{c}{3}\left(\frac{n-c}{c}\right)^3$. 
Then, inequality~\eqref{Equ:bound of clique number} holds for $6\le c \le 10$ is equivalent to show that $g(n,c)>0$ holds for $6\le c \le 10.$

For each $6\le c\le 10$, the function $g(n,c)$ and its derivative $g'(n,c)$ are as follows.
\vspace{-0.4cm}
\begin{align*} 
&g(n,6)   = \tfrac{1}{864}n^3 + \tfrac{3}{8}n^2 - \tfrac{21}{8}n + 5, 
& &g'(n,6)   = \tfrac{1}{288}n^2 + \tfrac{3}{4}n - \tfrac{21}{8} > 0 \ (n \geq 6). \\
&g(n,7)   = \tfrac{23}{3136}n^3 + \tfrac{479}{896}n^2 - \tfrac{297}{64}n + \tfrac{2735}{256}, 
& &g'(n,7)   = \tfrac{69}{3136}n^2 + \tfrac{479}{448}n - \tfrac{297}{64} > 0 \ (n \geq 7). \\ 
&g(n,8)   = \tfrac{1}{64}n^3 + \tfrac{21}{32}n^2 - 7n + \tfrac{305}{16},
& &g'(n,8)   = \tfrac{3}{64}n^2 + \tfrac{21}{16}n - 7 > 0 \ (n \geq 8). \\
&g(n,9)   = \tfrac{395}{15552}n^3 + \tfrac{283}{384}n^2 - \tfrac{615}{64}n + \tfrac{7791}{256}, 
& &g'(n,9)   = \tfrac{395}{5184}n^2 + \tfrac{283}{192}n - \tfrac{615}{64} > 0 \ (n \geq 9). \\
&g(n,10)  = \tfrac{29}{800}n^3 + \tfrac{31}{40}n^2 - \tfrac{99}{8}n + 45,
& &g'(n,10)  = \tfrac{87}{800}n^2 + \tfrac{31}{20}n - \tfrac{99}{8} > 0 \ (n \geq 10).  
\end{align*}
Thus, since $n\ge c,$ we have $g(n,c)\ge g(c,c)>0$ holds for $6\le c\le 10$, which completes the proof.
\QED

%\vspace{-0.3cm}
\subsection*{Appendix B: Proof of Theorem~\ref{Thm:Main} for $6 \leq n \leq 104$ with $n \not\in \{97,101\}$}

Recall that in inequality~\eqref{Equ:bound of clique number}, 
we can replace the first term $(\frac{n-5}{4})^4$ with $k_4(\turangraph{n-5}{4})$ because, using the induction hypothesis, we have $|\mathcal{C}_1'|\le k_4(\turangraph{n-5}{4})$.
Thus, we obtain that
\begin{equation*}
	\mathcal{C}\mathcal{C}_4(G)
    \leq 
    |\dC_1|
    \leq k_4(\turangraph{n-5}{4}) + 1+(n-5)
	+\frac{2}{5}(n-5)^2
	+\frac{2}{25}(n-5)^3.
\end{equation*}
Define $h(n): = k_4(\turangraph{n}{4})-k_4(\turangraph{n-5}{4}) - 1-(n-5)
-\frac{2}{5}(n-5)^2-\frac{2}{25}(n-5)^3$. 
It suffices to verify that $h(n)>0$ holds for $6\leq n\leq 104$ with $n\notin\{97,101\}$, 
which holds by the following case analysis.
\vspace{-0.4cm}
$$h(n)=
	\begin{cases}
	-\frac{1}{25}(3k^3 - 90k^2 + 75k - 50)>0\ \ \text{for}\ 0\le k\le 29, &\text{if}\ n=4k, \\
	-\frac{1}{25}(3k^3 - 74k^2 + 64k - 18)>0\  \ \text{for}\ 0\le k\le 23, &\text{if}\ n=4k+1,\\
	-\frac{1}{25}(3k^3 - 78k^2 + 51k - 14)>0\ \  \text{for}\ 0\le k\le 25, &\text{if}\ n=4k+2,\\
    -\frac{1}{25}(3k^3 - 82k^2 + 11k - 1)>0\ \, \ \ \text{for}\ 0\le k\le 27, &\text{if}\ n=4k+3.
	\end{cases}
$$
This completes the proof of Theorem~\ref{Thm:Main} for $6 \leq n \leq 104$ with $n \not\in \{97,101\}$.
\QED

\subsection*{Appendix C: Justification of the inequalities in the proof of Theorem~\ref{Thm:Main}.}

Here we provide detailed calculations to show that, in the proof of Theorem~\ref{Thm:Main},
$\coveropt{3}{H}\le \frac{59}{2}q^3 -24q^2d +\frac{13}{2}qd^2 -\frac{5}{9}d^3+Q_0<\frac{1}{64}(n-1)^3$ holds when $Q_0=Q_2$ or $Q_0=Q_3,$ where 
$Q_2=-23q^2+\frac{79}{6}qd-\frac{35}{18}d^2+\frac{11}{2}q-\frac{11}{6}d,$ and 
$Q_3= -\frac{105}{2}q^2+\frac{175}{6}qd-\frac{37}{9}d^2+\frac{92}{3}q-\frac{80}{9}d-\frac{16}{3}.$

Recall that $n\geq 105$ or $n\in \{97,101\}$, $\lfloor \frac{3}{4}n \rfloor +1\leq d \leq \frac{4}{5}n$, and from \eqref{Equ:range of q},
$\frac 1 4 d\le q\le p\le n-d<\frac 1 3 d.$
Thus, we have $q\geq \frac 1 4 (\lfloor \frac{3}{4}n \rfloor +1)\geq 18$.
If $Q_0=Q_2$, we get that
\begin{align*}\label{Equ:CC3(H)<F1(q,d)}
\coveropt{3}{H} \leq	
\frac{59}{2}q^3 -24q^2d +\frac{13}{2}qd^2 -\frac{5}{9}d^3
-23q^2+\frac{79}{6}qd-\frac{35}{18}d^2+\frac{11}{2}q-\frac{11}{6}d:=F(q,d). 
\end{align*}
Then $F'_d(q,d)=-24q^2+13qd-\frac 5 3 d^2+\frac{79}{6} q -\frac{35}{9}d -\frac{11}{6}.$
By~\eqref{Equ:range of q}, $3q<d\le 4q.$
It is not hard to see that
$F'_d(q,3q)=\frac 3 2 q - \frac {11} 6 >0,$ and
$F'_d(q,4q)=\frac 4 3 q^2-\frac {43}{18} q - \frac {11} 6>0$
holds for $q\geq 4.$
Therefore, $F(q,d)$ is increasing with respect to $d.$ 

If $q<\frac n 5,$
then $$F(q,d)\le F(q,4q)=\frac{35}{18}q^3-\frac{13}{9}q^2-\frac{11}{6}q<
\frac{7}{450}n^3-\frac{13}{225}n^2-\frac{11}{30}n<\frac{1}{64}(n-1)^3,$$
where the second inequality holds by the function $ F(q,4q)$ is increasing on $q$ when $q\ge1,$
and the last inequality holds when $n\ge 1.$

If $q\ge \frac n 5,$ since $d\le n-q,$ we have
\begin{equation*}\label{Equ:g1(q)}
F(q,d)\le F(q,n-q)=\frac{545}{9}q^3-\frac{116}{3}q^2n+\frac{49}{6}qn^2-\frac{5}{9}n^3-\frac{343}{9}q^2+\frac{307}{18}qn-\frac{35}{18}n^2+\frac{22}{3}q-\frac{11}{6}n.
\end{equation*}
Let the result equation be $g(q).$ 
Then $g'(q)=\frac{545}{3}q^2-\frac{232}{3}qn+\frac{49}{6}n^2-\frac{686}{9}q+\frac{307}{18}n+\frac{22}{3}.$
Since $\frac n 5\le q\le n-\lfloor\frac{3}{4}n\rfloor-1=\lceil\frac{n}{4}\rceil-1<\frac n 4,$
it is easy to see that 
$g'(0)=\frac{49}{6}n^2 + \frac{307}{18}n + \frac{22}{3}>0$ for $n\ge 0$,
$$g'(\frac n 5)=-\frac{1}{30}n^2+\frac{163}{90}n+\frac{22}{3}<0\ \  \text{and}\ \ 
g'(\frac n 4)=\frac{3}{16}n^2 - 2n +\frac{22}{3}>0,$$ 
when $n\ge 59.$
Thus, $g(q)$ is decreasing first and then increasing with respect to $q$ 
and the maximum of $g(q)$ will take value at the endpoints of the interval containing $q,$ i.e., 
$g(q)\le \max\{g(\frac n 5),g(\lceil\frac{n}{4}\rceil-1)\}.$
Note that 
$g(\frac n 5)=F(\frac n 5,\frac {4n} 5)=\frac{7}{450}n^3-\frac{13}{225}n^2-\frac{11}{30}n<\frac{1}{64}(n-1)^3,$ when $n\ge 1.$
And 
$g(\frac{n}{4})=\frac{1}{64}n^3-\frac{1}{16}n^2<\frac{1}{64}(n-1)^3,$
when $n\ge 1.$
Therefore,  when $Q_0=Q_2$, $\coveropt{3}{H}<\frac{1}{64}(n-1)^3$ and
$\coveropt{4}{G}\le \coveropt{4}{G'}+\coveropt{3}{H}<\coveropt{4}{\turangraph{n}{4}},$ 
and we are done.

If $Q_0=Q_3$, we get that
\begin{align*}\label{Equ:CC3(H)<F2(q,d)}
\coveropt{3}{H} \leq	
\frac{59}{2}q^3 -24q^2d +\frac{13}{2}qd^2 -\frac{5}{9}d^3
-\frac{105}{2}q^2+\frac{175}{6}qd-\frac{37}{9}d^2+\frac{92}{3}q-\frac{80}{9}d-\frac{16}{3}:=F(q,d). 
\end{align*}
Then $F'_d(q,d)=-24q^2+13qd-\frac 5 3 d^2+\frac{175}{6} q -\frac{74}{9}d -\frac{80}{9}.$
$F'_d(q,3q)=\frac 9 2 q - \frac {80} 9 >0,$ and
%$F'_d(q,3q+1)=\frac{9}2 q-\frac{133}{18}>0,$ 
$F'_d(q,4q)=\frac 4 3 q^2-\frac {55}{18} q - \frac {80} 9>0$
holds for $q\geq 4.$
Therefore, $F(q,d)$ is increasing with respect to $d.$ 

If $q<\frac n 5,$
then $$F(q,d)\le F(q,4q)=\frac{35}{18}q^3-\frac{29}{18}q^2-\frac{44}{9}q-\frac {16} 3<
\frac{7}{450}n^3-\frac{29}{450}n^2-\frac{44}{45}n-\frac {16} 3<\frac{1}{64}(n-1)^3,$$
where the second inequality holds by the function $ F(q,4q)$ is increasing on $q$ when $q\ge2,$
and the last inequality holds for all $n\ge 0.$

If $q\ge \lceil \frac n 5 \rceil,$ since $d\le n-q,$ we have
\begin{equation*}\label{Equ:g2(q)}
F(q,d)\le F(q,n-q)=\frac{545}{9}q^3-\frac{116}{3}q^2n+\frac{49}{6}qn^2-\frac{5}{9}n^3-\frac{772}{9}q^2+\frac{673}{18}qn-\frac{37}{9}n^2+\frac{356}{9}q-\frac{80}{9}n-\frac {16} 3.
\end{equation*}
Let the result equation be $g(q).$ 
Then $g'(q)=\frac{545}{3}q^2-\frac{232}{3}qn+\frac{49}{6}n^2-\frac{1544}{9}q+\frac{673}{18}n+\frac{356}{9}.$
Since $\lceil \frac n 5 \rceil\le q\le n-\lfloor\frac{3}{4}n\rfloor-1=\lceil\frac{n}{4}\rceil-1<\frac n 4,$
it is easy to see that 
$g'(0)=\frac{49}{6}n^2+\frac{673}{18}n+\frac{356}{9}>0$ for $n\ge 0$,
$$g'(\frac n 5)=-\frac{1}{30}n^2+\frac{277}{90}n+\frac{356}{9}<0 \ \text{for}\ n\ge 104, \ \text{and}\  g'(\frac n 4)=\frac{3}{16}n^2-\frac{11}{2}n+\frac{356}{9}>0\ \text{for}\ n\ge 17.$$
Thus, if $n\ge 105$,  
$g(q)$ is decreasing first and then increasing with respect to $q$
and the maximum of $g(q)$ will take value at the endpoints of the interval containing $q,$ i.e., 
$g(q)\le \max\{g(\frac n 5),g(\frac{n}{4})\}$.
Note that 
$g(\frac n 5)=F(\frac n 5,\frac {4n} 5)=\frac{7}{450}n^3-\frac{29}{450}n^2-\frac{44}{45}n-\frac {16} 3<\frac{1}{64}(n-1)^3$ for all $n\ge 0,$
and 
$g(\frac{n}{4})=F(\frac n 4,\frac {3n} 4)=\frac{1}{64}n^3-\frac{1}{8}n^2+n-\frac{16}{3}<\frac{1}{64}(n-1)^3$
for all $n\ge 0$.

If $n=97$, then
$g'(\lceil \frac n 5 \rceil)=g'(20)<0$
%=\frac{545}{3}20^2-\frac{232}{3}97\times 20 +\frac{49}{6}97^2-\frac{1544}{9}20+\frac{673}{18}97+\frac{356}{9}<0$, 
and $g(q)\le \max\{g(\lceil \frac n 5 \rceil), g(\lceil\frac{n}{4}\rceil-1)\}=\max\{g(20), g(24)\}=13406<13824=\frac{(97-1)^3}{64}$.

If $n=101$, then 
$g'(\lceil \frac n 5 \rceil)=g'(21)<0$
%=\frac{545}{3}21^2-\frac{232}{3}101\times 21 +\frac{49}{6}101^2-\frac{1544}{9}21+\frac{673}{18}101+\frac{356}{9}<0$, 
and $g(q)\le \max\{g(\lceil \frac n 5 \rceil), g(\lceil\frac{n}{4}\rceil-1)\}=\max\{g(21), g(25)\}=15099<15625=\frac{(101-1)^3}{64}$.

Therefore, when $Q_0=Q_3$, $\coveropt{3}{H}<\frac{1}{64}(n-1)^3$ and 
$\coveropt{4}{G}\le \coveropt{4}{G'}+\coveropt{3}{H}<\coveropt{4}{\turangraph{n}{4}},$ 
which completes the proof of Theorem~\ref{Thm:Main}.
\QED

\subsection*{Appendix D: Justification of the inequalities in the proof of Lemma~\ref{Lem:adjustment method}.}

Here we show that if $A'$ is of Type 2 (Type 3), then $f(A')= 28q^3-\frac{45}{2}q^2m+6qm^2-\frac{1}{2}m^3+R_2 (R_3),$
where
$R_2=-28q^2+\frac{33}{2}qm-\frac{5}{2}m^2+6q-2m$ and 
$R_3=-56q^2+\frac{63}{2}qm-\frac{9}{2}m^2+34q-10m-6$.

If $A'$ is of Type 2, 
then $A'=(4,...4,3,...,3,2).$ 
We can see that there are $(m-3q+1)$ many $4$'s and $(4q-m-2)$ many $3$'s. 
To calculate $f(A'),$
by Definition~\ref{Dfn:greedy sequence}, 
plugging into $a=m-3q+1,$ $b=q-1$, $c=q$ and $p=q$, 
we get that
$S_1 (A') =b= q-1$, 
$S_2(A') = 12q^2-6qm+m^2-9q+m+3$,
and
\begin{align*}
S_3(A')
=& \sum_{k=3}^{m-3q+1} 4\times 3\sum_{j=2}^{k-1} (j-1)
+\sum_{k=m-3q+2}^{q-1} 3 \times 3\sum_{j=2}^{m-3q+1} (j-1)
+\sum_{k=m-3q+3}^{q-1} 3 \times 2 \sum_{j=m-3q+2}^{k-1} (j-1)\\
&+2\left(3\sum_{j=2}^{m-3q+1}(j-1)+2\sum_{m-3q+2}^{q-1}(j-1)\right)\\ 
&=28q^3 - \frac{45}{2}q^2m +6qm^2-\frac{1}{2}m^3- 40q^2+\frac{45}{2}qm -\frac{7}{2}m^2+14q-3m-2.&
\end{align*}
Thus, we have 
$ f(A')
=S_1(A')+S_2(A')+S_3(A')
=28q^3 - \frac{45}{2}q^2m +6qm^2-\frac{1}{2}m^3
+R_2,
$
where $R_2=-28q^2+\frac{33}{2}qm-\frac{5}{2}m^2+6q-2m.$

If $A'$ is of Type 3, 
then $A'=(4,...4,3,...,3,1).$ 
We can see that there are $(m-3q+2)$ many $4$'s and $(4q-m-3)$ many $3$'s. 
To calculate $f(A'),$
by Definition~\ref{Dfn:greedy sequence}, 
plugging into $a=m-3q+2,$ $b=c=q-1$ and $p=q$, 
we have
$S_1 (A') =b= q-1$, 
$S_2(A') = 12q^2-6qm+m^2-17q+3m+7$, and
\begin{align*}
S_3(A')
=& \sum_{k=3}^{m-3q+2} 4\times 3\sum_{j=2}^{k-1} (j-1)
+\sum_{k=m-3q+3}^{q-1} 3 \times 3\sum_{j=2}^{m-3q+2} (j-1)
+\sum_{k=m-3q+4}^{q-1} 3 \times 2 \sum_{j=m-3q+3}^{k-1} (j-1)\\
&+3\sum_{j=2}^{m-3q+2}(j-1)+2\sum_{m-3q+3}^{q-1}(j-1)\\ 
&=28q^3 - \frac{45}{2}q^2m +6qm^2-\frac{1}{2}m^3- 68q^2+\frac{75}{2}qm -\frac{11}{2}m^2+50q-13m-12&.
\end{align*}
Thus, we have 
$ f(A')
=S_1(A')+S_2(A')+S_3(A')
=28q^3 - \frac{45}{2}q^2m +6qm^2-\frac{1}{2}m^3
+R_3,
$
where $R_3=-56q^2+\frac{63}{2}qm-\frac{9}{2}m^2+34q-10m-6.$
This completes the proof of Lemma~\ref{Lem:adjustment method}.
\QED

\subsection*{Appendix E: Justification of the inequality~\eqref{Equ:best bound of CC3(H)}}
Here we show that if $A'$ is of Type 2 (Type 3), then
$\coveropt{3}{H}\le \frac{59} 2 q^3 - 24 q^2 m + \frac{13}{2} q m^2 - \frac{5}{9} m^3+Q_2 (Q_3)$,
where $Q_2=-23q^2+\frac{79}{6}qm-\frac{35}{18}m^2+\frac{11}{2}q-\frac{11}{6}m,$ and
$Q_3= -\frac{105}{2}q^2+\frac{175}{6}qm-\frac{37}{9}m^2+\frac{92}{3}q-\frac{80}{9}m-\frac{16}{3}.$

If $A'$ is of Type 2,
then $a=m-3q+1,$ and $R_0=R_2=-28q^2+\frac{33}{2}qm-\frac{5}{2}m^2+6q-2m.$
By~\eqref{Equ:improvement when E is small}, we get that
\begin{align*}
\coveropt{3}{H}
\leq&
28q^3-\frac{45}{2}q^2m+6qm^2-\frac{1}{2}m^3
-28q^2+\frac{33}{2}qm-\frac{5}{2}m^2+6q-2m-\frac{2}{9}(m-3q)^2(m-3q-1)\\
&+\frac{1}{6}(m-3q+1)^2(m-3q)=\frac{59} 2 q^3 - 24 q^2 m + \frac{13}{2} q m^2 - \frac{5}{9} m^3
   +Q_2,
\end{align*}
where $Q_2=-23q^2+\frac{79}{6}qm-\frac{35}{18}m^2+\frac{11}{2}q-\frac{11}{6}m.$

If $A'$ is of Type 3,
then $a=m-3q+2,$ and $R_0=R_3=-56q^2+\frac{63}{2}qm-\frac{9}{2}m^2+34q-10m-6.$
By~\eqref{Equ:improvement when E is small}, we get that
\begin{align*}
\coveropt{3}{H}
\leq&
28q^3-\frac{45}{2}q^2m+6qm^2-\frac{1}{2}m^3
-56q^2+\frac{63}{2}qm-\frac{9}{2}m^2+34q-10m-6
-\frac{2}{9}(m-3q+1)^2\cdot\\
&\cdot(m-3q)
+\frac{1}{6}(m-3q+2)^2(m-3q+1)
=\frac{59} 2 q^3 - 24 q^2 m + \frac{13}{2} q m^2 + \frac{5}{9} m^3
   +Q_3,
\end{align*}
where $Q_3= -\frac{105}{2}q^2+\frac{175}{6}qm-\frac{37}{9}m^2+\frac{92}{3}q-\frac{80}{9}m-\frac{16}{3},$
%Therefore, combining with the inequality when $A'$ is of Type 1, we have
%$$\coveropt{3}{H}\le \frac{59} 2 q^3 - 24 q^2 m + \frac{13}{2} q m^2 - \frac{5}{9} m^3+Q_0,$$
%where $Q_0=\max\{Q_1,Q_2,Q_3\},$ 
%and 
which completes the proof of~\eqref{Equ:best bound of CC3(H)}.
\QED

\end{document}